\documentclass{amsart}
\usepackage[utf8]{inputenc}
\usepackage{amsmath, amssymb, amsthm}
\usepackage{mathrsfs}
\usepackage[margin=3.5cm]{geometry}
\usepackage[colorlinks]{hyperref}
\usepackage{comment}
\usepackage{graphicx}
\usepackage{enumitem}

\usepackage[backend = bibtex,style = numeric, sorting=nyt]{biblatex}

\makeatletter
\let\oldhypertarget\hypertarget
\renewcommand{\hypertarget}[2]{%
  \oldhypertarget{#1}{#2}%
    \protected@write\@mainaux{}{%
        \string\expandafter\string\gdef
          \string\csname\string\detokenize{#1}\string\endcsname{#2}%
    }%
  }
\newcommand{\myhyperlink}[1]{%
  \hyperlink{#1}{\csname #1\endcsname}%
  }
\makeatother

\newcommand{\R}{\mathbb{R}}
\renewcommand{\S}{\mathbb{S}}

\newcommand{\s}{u}

\newtheorem{thm}{Theorem}[section]
\newtheorem{lem}[thm]{Lemma}
\newtheorem{prop}[thm]{Proposition}

\newtheorem{rmk}[thm]{Remark}

\title{A curvature flow that deforms curves to an embedded target}

\author{Samuel Cuthbertson\and Glen Wheeler\and Valentina Wheeler}

\begin{document}

\begin{abstract}
In this paper we introduce the \emph{target flow}---a specific curve shortening flow with an ambient forcing term---that, given an embedded (not necessarily convex) target curve $\eta:\S\to\R^2$, will attempt to evolve a given source curve to that target.
The motivation for this flow is to address a question of Yau.
Our main result is that the target flow with uniformly normal graphical data converges smoothly to the target, broadening the class of known sources and targets such that Yau's problem has a solution.
\end{abstract}

\maketitle

\section{Introduction}

Let $\eta:\S^1\rightarrow \R^2$ be an embedded smooth curve. 
Our main object of study is the curvature flow
\begin{equation}
\label{EQflow}
\partial_t \gamma
 = \kappa + V^\eta(\gamma)
\end{equation}
where $\gamma:\S\times[0,T)\rightarrow\R^2$ is a one-parameter family of evolving curves, $\kappa = k\nu$ is the curvature vector, and $V^\eta:\R^2\to\R^2$ is defined as follows.
In the maximal uniform tubular neighbourhood $T^\eta\subset\R^2$ of $\eta(\S)$, with normal coordinates $(\s,r)\mapsto (\eta(\s), \eta(\s)+r\nu^\eta(\s))$ (where $\s$ and $\nu^\eta$ are the arc-length parameter of $\eta$ and the unit normal to $\eta$ respectively), we define
\[
V^\eta(\eta(\s)+r\nu^\eta(\s))
 := 
	\left(-
		\frac{k^\eta}{1-k^\eta r}-C(\eta)r
	\right)\nu^{\eta}(\s)
\,.
\]
We call $\gamma_0(\cdot) = \gamma(\cdot,0)$ the \emph{source curve} and $\eta$ the \emph{target curve}.
The constant $C(\eta)$ depends on the target only and is determined in Section \ref{Ssmoothconv}, specifically Propositions \ref{pro: rs} and \ref{PRrssunif}.

The flow \eqref{EQflow} will not move a (normal graphical) source curve from inside $T^\eta$ to the outside $\R^2\setminus T^\eta$, and so the definition of $V^\eta$ on $\R^2\setminus T^\eta$ is not important.
We set $V^\eta(x,y) = 0$ for $(x,y)\in\R^2\setminus T^\eta$.

The flow \eqref{EQflow} is stationary if the evolving curve $\gamma$ is equal to the target $\eta$.
Furthermore, the flow moves to match the curvature scalars and position vectors of $\gamma(\cdot,t)$ and $\eta$.
Our main result is that the flow \eqref{EQflow} drives all initial data $\gamma_0$ that is normal graphical over the target $\eta$ to $\eta$ smoothly in infinite time.

\begin{thm}\label{TMmain}
Let $\eta:\S\rightarrow \R^2$ be an embedded smooth curve, and $\gamma_0:\S\to\R^2$ a curve with $\gamma_0(\S)\subset T^\eta$, that is, normal graphical over $\eta$.
Assume further that $r_0\in h^{2+\alpha}([0,L(\eta)])$, that is, the initial graph function $r_0$ is in the $(2+\alpha)$ little H\"older space.

Then, the solution $\gamma:\S\times[0,T)\to\R^2$ to the target flow \eqref{EQflow} exists uniquely and has the following properties:
\begin{enumerate}[label=(\roman*)]
\item $\gamma(\cdot,t)$ remains normal graphical,
\item is smooth (or analytic if $\eta$ is analytic),
\item the maximal time of existence is infinite, and
\item $\gamma(\cdot,t)$ converges exponentially fast in $C^\infty$ to $\eta$.
\end{enumerate}
\end{thm}

\begin{figure}[ht]
\centering
\begin{minipage}{0.45\textwidth}
    \centering
    \includegraphics[width=\textwidth]{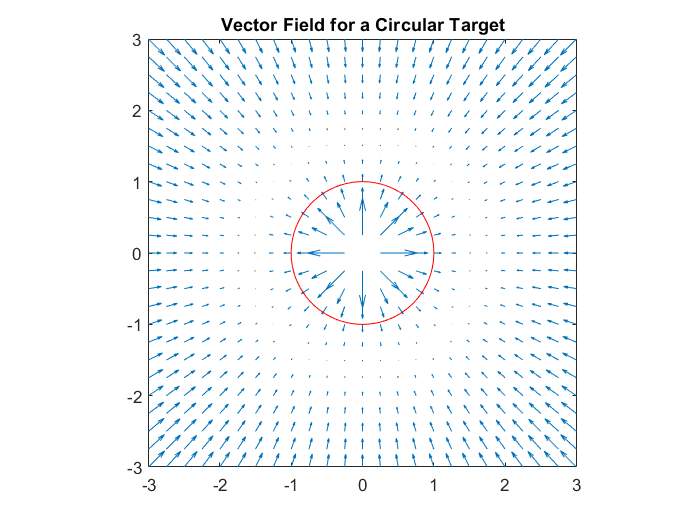}
\end{minipage}
\hspace{-0.05\textwidth}
\begin{minipage}{0.45\textwidth}
    \centering
    \includegraphics[width=\textwidth]{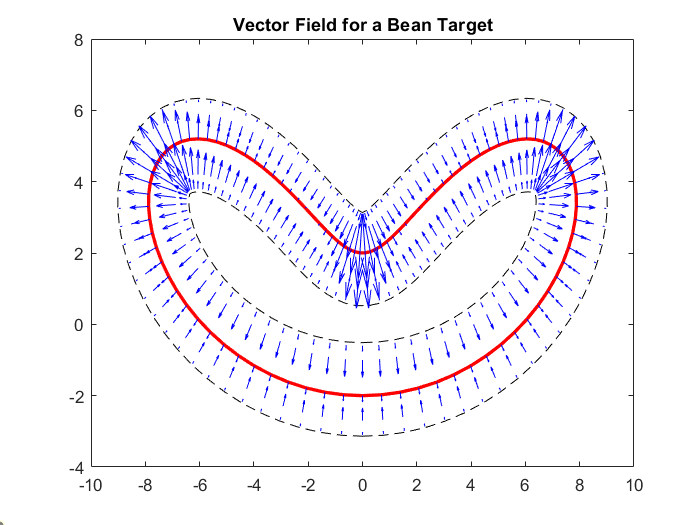}
\end{minipage}
\caption{
(Left) A plot of the vector field $V^\eta$ for a circle of radius $1$.
Arrows have been scaled by $1.5$ for visual effect.
(Right) A plot of the vector field $V^\eta$ in the tubular neighbourhood of a bean parametrised by $\eta(u) = (x(u),y(u))$ with $x(u) = 7\sin(u) + 2\sin(2u)$, and $y(u) = 4 - 2\cos(u) - 4\cos(u)^2 + \sin^2(u)$.
Note that (1) on the target, the vector field is $-k^\eta\nu^\eta$, and (2) some vectors are very large (especially near points of high curvature) on one side, which can obscure the smaller vectors.}
\end{figure}

\begin{rmk}
The little H\"older space $h^{2+\alpha}$ is the closure of smooth functions in the standard $C^{2+\alpha}$ norm.
It is a subset of $C^{2+\alpha}$, and has improved compactness and approximation properties over the standard space.
We are using it here as it is a natural setting in which to apply Angenent's linearisation approach to obtain existence and uniqueness (see Section \ref{Sexistuniq}).
The atypical degeneracy of the target flow is the primary technical obstacle preventing us from weakening this assumption.
\end{rmk}

Our main motivation for studying this flow is to resolve Yau's problem for
broader classes of source and target curves.
Yau's problem asks if, given a target curve $\eta$, it is possible to find a
parabolic flow that evolves as many source curves as possible to the target.
Our flow addresses this directly as the velocity of \eqref{EQflow} depends only on $\eta$.
Thus, Theorem \ref{TMmain} implies that Yau's problem is solvable for embedded smooth target curves, and normal graphical source curves of regularity class $h^{2+\alpha}$.

There are a number of recent results on Yau's problem, which we now briefly survey.
Lin and Tsai considered \cite{lin2009evolving} the flow
\begin{align*}
\gamma_t = -\left(\frac{1}{k}-\frac{1}{k^\eta}\right)\nu\,,
\end{align*}
in the angle parametrisation $(\theta,t)$.
Their analysis requires convexity and that the initial curve and the target
have the same length.
The advantage of this flow is that the difference of inverse curvature,
$z(\theta,t) = \frac{1}{k^\gamma(\theta,t)}-\frac{1}{k^\eta(\theta,t)} $
satisfies the linear equation
\[
z_t = z_{\theta \theta}+z\,.
\]
This makes the convergence argument more straightforward, and the authors are
able to show that convex source curves converge to convex targets.

Gage-Li \cite{gage1993evolving,gage1994evolving} studied convex curves
evolving under the equation
\[
\gamma_t = f(\theta)k\nu
\,.
\]
They showed that if the function $f$ took the form
\[
f(\theta) = \frac{-h^\eta(\theta)}{k^\eta(\theta)}
\]
where $h^\eta = \langle \eta, (\cos(\theta),\sin(\theta))\rangle,$ the support
function of the target curve of some smooth, symmetric convex curve $\eta$,
then the flow converges in shape to $\eta$ as it shrinks to a point.

In an important contribution, Chou-Zhu \cite{chou1999anisotropic} studied the anisotropic flow of convex curves:
\[
\beta(\theta)\gamma_t = (g(\theta)k+F)\nu\,.
\]
Here both $\beta$ and $g$ are positive and $F$ is a constant.
They proved that there exists a negative $F^*$ such that the curve converges in shape to a solution of
\[
g(\theta)k(\theta) 
 + F^*+\langle c^*,(\cos(\theta),\sin(\theta))\rangle \beta(\theta) = 0
\,,
\]
where $c^* = (c_1^*,c_2^*)$ is the unique point satisfying
\[
\int_0^{2\pi}\frac{g(\theta)e^{i\theta}d\theta}{\beta(\theta)(c_1^*\cos(\theta)+c_2^*\sin(\theta))+F} = 0
\,.
\]
Gao-Zhang \cite{gao2019yau} suggested the flow
\begin{equation}
\gamma_t = \left( k-\lambda(t)k^\eta \right)\nu
\,.
\end{equation}
They choose
\[
\lambda(t) = \frac{2\pi}{\int_{\mathbb{S}^1}k^\eta|\gamma_u|du}
\,.
\]
The quantity $\lambda$ ensures that the evolving curve has constant area. 
Their analysis uses convexity as well, and they proved that the flow converges
to a curve congruent to $\sqrt{\frac{A[\gamma_0]}{A[\eta]}}\eta$.

The main advantage of our flow \eqref{EQflow} over the existing literature is
that we allow for (1) non-convex source curves, (2) non-convex target curves, and (3) we do not require any action that depends on the initial curve in order for the position vector of the source to converge to the target (for example, a translation of the final image).
The disadvantages with our approach are that (a) we do require relatively high regularity of the source in order to overcome the degeneracy of the flow and obtain uniqueness of the solution, and (b) we require the normal graphical condition on the source curve.

We prove Theorem \ref{TMmain} in three main steps.
In Section \ref{Spre}, we set our notation and derive the evolution equations for the graph function.
First, in Section \ref{Spre} we introduce our notation, setting, and derive the graph flow.
Then, in Section \ref{Sapriori}, we establish the existence of appropriate barriers and prove a-priori estimates using maximum principle arguments. 
Second, in Section \ref{Sexistuniq}, we prove global existence and uniqueness, which is non-standard due to the strong degeneracy in the speed of the flow once written in graphical coordinates; to deal with this, we use ideas due to Angenent. 
Third, in Section \ref{Ssmoothconv}, we prove explicit decay estimates for the normal distance to the target curve, in order to conclude smooth convergence.

\section{Preliminaries}
\label{Spre}

\subsection{Evolution equation for the graph function with arbitrary \texorpdfstring{$V$}{V}}

Let $\eta:[0,L^\eta]\rightarrow \R^2$ be a given embedded curve parametrised by arclength $\s$.
We work only with curves $\gamma$ that can be parametrised by
\begin{equation}\label{EQnormgraph}
\gamma(\s) = \eta(\s) + r(\s)\nu^\eta
\,,
\end{equation}
where $r:[0,L^\eta)\rightarrow \R^2$.
Such a curve will be referred to as \emph{normal graphical} and the function $r$ is the \emph{graph function}.
This is a standing assumption throughout the paper (we will preserve normal graphicality in Section \ref{Sapriori}), and so when we consider solutions to \eqref{EQflow}, we are assuming also that these solutions can be parametrised according to \eqref{EQnormgraph}.

It is important to note that, as we use the inward-pointing unit normal $\nu^\eta$ in \eqref{EQnormgraph}, \emph{positive} values of $r$ correspond to the curve $\gamma$ being inside $\eta$, and \emph{negative} values of $r$ correspond to $\gamma$ being outisde $\eta$.
If $\eta$ is convex, the uniform tubular neighbourhood in which a curve can be normal-graphical is infinite on one side.
Set
\[
I^\eta =
\begin{cases}
\left(-\infty, \frac{1}{k_{\max}}\right) & \text{ if }\eta \text{ is convex }\\
\left(\frac{1}{k_{\min}}, \frac{1}{k_{\max}}\right) & \text{ if }\eta \text{ is  nonconvex. }
\end{cases}
\]
The maximal uniform tubular neighbourhood of $\eta(\S)$ in $\R^2$, called earlier $T^\eta$, is the following strip around $\eta$:
\[
T^\eta = \{ p\in\R^2\,:\, p = \eta(\s) + r\nu^\eta(\s)\,,\text{ for some $\s\in[0,L^\eta]$, $r\in I^\eta$}\}
\,.
\]
There is a map between functions $r:[0,L^\eta]\to\R$ and curves $\gamma^r:[0,L^\eta]\to\R^2$.
Given a function $r$ such that $r(\s)\in I^\eta$ for all $\s$, the curve $\gamma^r$ \emph{generated by $r$} is that given by the formula \eqref{EQnormgraph}, that is,
\[
\gamma^r(\s) = \eta(\s) + r(\s)\nu^\eta
\,.
\]

Let us now consider a one-parameter family of curves $\gamma:[0,L^\eta]\times[0,T)\to\R^2$, all of which are normal graphical (satisfying \eqref{EQnormgraph} with $r = r(\s,t)$).
This family evolves according to the target flow \eqref{EQflow} if and only if the function $r$ evolves according to a specific non-linear parabolic PDE, which we shall now derive.

Differentiating \eqref{EQnormgraph} gives
\begin{align*}
\gamma_\s &= \tau^\eta+r_\s\nu^\eta-k^\eta r\tau^\eta
 = (1-k^\eta r)\tau^\eta+r_\s\nu^\eta
\,,\text{ so }
\\
|{\gamma}_\s| &= \sqrt{(1-k^\eta r)^2+(r_\s)^2}\,.
\intertext{Thus the tangent and normal vectors along $\gamma$ are given by}
{\tau} &= \frac{(1-k^\eta r)\tau^\eta + r_\s\nu^\eta}{ \sqrt{(1-k^\eta r)^2+(r_\s)^2}}
\qquad\text{and}\qquad
{\nu} = \frac{(1-k^\eta r)\nu^\eta-r_\s\tau^\eta}{ \sqrt{(1-k^\eta r)^2+(r_\s)^2}}\,.
\intertext{Continuing to the curvature of $\gamma$, we calculate}
{\gamma}_{\s\s} &= k^\eta (1-k^\eta r)\nu^\eta - (k^\eta)_\s r\tau^\eta -2k^\eta r_\s\tau^\eta+r_{\s\s}\nu^\eta
\\
&= (r_{\s\s}+k^\eta (1-k^\eta r))\nu^\eta-((k^\eta)_\s r+2k^\eta r_\s)\tau^\eta
\,,\text{ so }
\\
{k} &= \frac{(1-k^\eta r)(r_{\s\s}+k^\eta (1-k^\eta r))+r_\s((k^\eta)_\s r+2k^\eta r_\s)}{((1-k^\eta r)^2+(r_\s)^2)^{\frac{3}{2}}}
\\
&= \frac{(1-k^\eta r)r_{\s\s}+k^\eta\left( (1-k^\eta r)^2+(r_\s)^2\right) +r_\s(k^\eta r)_\s}{((1-k^\eta r)^2+(r_\s)^2)^{\frac{3}{2}}}
\\
&= \frac{(1-k^\eta r)r_{\s\s}+ k^\eta r_\s^2 + r (k^\eta)_\s}{((1-k^\eta r)^2+(r_\s)^2)^{\frac{3}{2}}}
	+ \frac{k^\eta}{((1-k^\eta r)^2+(r_\s)^2)^{\frac{1}{2}}}
\,.
\end{align*}
Note that also
\begin{equation}
\label{EQcosanglenormals}
\nu^\eta\cdot\nu = \frac{1-k^\eta r}{((1-k^\eta r)^2+(r_\s)^2)^{\frac{1}{2}}}
\,.
\end{equation}
Thus $\gamma$ is a curve shortening flow with ambient force field $V$, $\gamma_t = \kappa + V(\gamma)$, if and only if
\begin{align}
r_t
	&= \frac1{\nu^\eta\cdot\nu}\left(
		k + V(\gamma)\cdot\nu
	\right)
\notag\\
	&= \frac{((1-k^\eta r)^2+(r_\s)^2)^{\frac{1}{2}}}{1-k^\eta r}\left(
   		\frac{(1-k^\eta r)r_{\s\s}+ k^\eta r_\s^2 + r (k^\eta)_\s}{((1-k^\eta r)^2+(r_\s)^2)^{\frac{3}{2}}}
		+ \frac{k^\eta}{((1-k^\eta r)^2+(r_\s)^2)^{\frac{1}{2}}}
		+ V(\gamma)\cdot\nu
	\right)
\notag\\
	&= \frac{((1-k^\eta r)^2+(r_\s)^2)^{\frac{1}{2}}}{1-k^\eta r}\left(
   		\frac{(1-k^\eta r)r_{\s\s}+ k^\eta r_\s^2 + r (k^\eta)_\s}{((1-k^\eta r)^2+(r_\s)^2)^{\frac{3}{2}}}
		+ \frac{k^\eta}{((1-k^\eta r)^2+(r_\s)^2)^{\frac{1}{2}}}
		+ V(\gamma)\cdot\nu
	\right)
\notag\\
&= 
	\frac{r_{\s\s}}{(1-k^\eta r)^2+(r_\s)^2}
	+ \frac{k^\eta r_\s^2 + (k^\eta)_\s r}{((1-k^\eta r)^2+(r_\s)^2)(1-k^\eta r)}
	+ \frac{k^\eta}{1-k^\eta r} 
\notag\\
&\hspace{5cm}
	+ \frac{\sqrt{(1-k^\eta r)^2+(r_\s)^2}}{(1-k^\eta r)} V(\eta+r\nu^\eta)\cdot \nu
\label{EQrflow}
\end{align}
for $(\s,t)\in [0, L(\eta))\times(0,T)$, with initial condition $r(\s,0) = r_0$ for $\s\in[0,L(\eta))$.

\subsection{Derivation of \texorpdfstring{$V$}{V}}
\label{SderivV}

We will define $V$ to be parallel to normal rays of $\eta$. In particular, for some function $f(r,k^\eta)$ we set
\[
V(\eta+r\nu^\eta) = f(r,k^\eta)\nu^\eta.
\]
Observe that we must have $f(0,k^\eta) = -k^\eta$ so that $V(\eta) = -k^\eta\nu^\eta$.
This ensures that the flow stops once the solution curve agrees with that of the target.

Let $r$ be a solution to the flow equation \eqref{EQflow} that is constant in space, that is, $r(\s,t) = \hat r(t)$.
Then (recall \eqref{EQcosanglenormals}) $\nu^\eta\cdot\nu = 1$ and
\[
\hat r_t = \frac{k^\eta}{1-k^\eta \hat r} + f(\hat r,k^\eta).
\]
We wish to have exponential decay of $\hat r$, which is ensured by a differential equation of the form $\hat r_t = -C\hat r$ where $C$ is some constant.

We guarantee that this holds by choosing, in general,
\[
f(r,k^\eta) = -Cr-\frac{k^\eta}{1-k^\eta r}
\,.
\]
We leave the constant $C>0$ as a parameter for now (it will be set to a constant depending only on $\eta$ in Section 5).

With this choice, we have at our disposal solutions $\hat r(t) = \hat r(0) e^{-Ct}$, which will be very useful in our analysis.

  \subsection{Evolution equation for the graph function of the target flow}

Incorporating the definition of $V$ given above, we find that $\gamma(\s,t)$ is a target flow \eqref{EQflow} if and only if 
\begin{equation}\label{PDE r}
r_t
 = \frac{r_{\s\s}}{(1-k^\eta r)^2+(r_\s)^2} 
	+ \frac{r_\s(k^\eta r)_\s}{((1-k^\eta r)^2+(r_\s)^2)(1-k^\eta r)}
	- Cr
\,,
\quad
\text{$r(\cdot,t)$ periodic}\,,
\quad
r(\cdot,0) \in I^\eta\,.
\end{equation}
Note that if the initial data $r(\cdot,0)$ does not have values contained in $I^\eta$, then the curve generated by $r(\cdot,0)$ will not be contained in the uniform tubular neighbourhood $T^\eta$.
We do not consider such initial data.

We assume that $r_0\in h^{2+\alpha}([0,L^\eta])$.
For each fixed $t$, the solution curve $\gamma(\cdot,t)$ generated by $r(\cdot,t)$ belongs to the space
\[
X^\eta  = \{ \gamma \in h^{2+\alpha}(\S^1,\R^2)\,:\, \gamma(\S^1) \subset \overline{T^\eta} \text{ and is normal graphical}\}.
\]
In fact, much more than this is true: The solution will be smooth (or analytic) for $t>0$.
This uses smoothness of $\eta$ (or analyticity).

From here on we will remove the $\eta$ superscript unless needed to avoid confusion. 
For convenience we use the abbreviation $v = (1-kr)^2+(r_\s)^2$.


\section{A-priori estimates}
\label{Sapriori}

First, we take advantage of the barriers we mentioned earlier.

\begin{lem}(Comparison with barriers)
Suppose there are constants $r_1\leq 0\le r_2$ such that  $r_1<r(\s,0)< r_2$ for all $\s \in [0,L(\eta))$, and that $r_1, r_2, r(\cdot,t) \in I$.
Then any solution to \eqref{PDE r} with initial data $r(\cdot,0)$ satisfies
\[
 r_1e^{-Ct}
 < r(\s,t)
 < r_2e^{-Ct}
\,.
\]
\label{LMrbd}
\end{lem}
\begin{proof}
First, our hypothesis implies that the curves generated by $r_1$ and $r_2$ are normal graphical and admissible for the flow.
The solution curves are generated by $r_ie^{-Ct}$ (see the calculation in Section \ref{SderivV}).

Let $\varepsilon>0$ be a parameter.
Set $\tilde{r} = r_2e^{-C(1-\varepsilon)t}$ and $w(\s,t) = \tilde{r}-r(\s,t)$.
Suppose for contradiction that there is a first point $(\s^*,t^*)$ such that $w(\s^*,t^*) = 0$. Since such a point is a new minimum we have 
\begin{align*}
w_t(\s^*,t^*) &\leq 0\\
w_\s(\s^*,t^*) & = 0\\
w_{\s\s}(\s^*,t^*)& \geq 0.
\end{align*}
The evolution equation \eqref{PDE r} then implies, at $(u^*,t^*)$, 
\begin{align*}
0\geq w_t& = -C(1-\varepsilon)r_2e^{-C(1-\varepsilon)t}-r_t\\
&=-C(1-\varepsilon)\tilde{r}-\frac{r_{\s\s}}{v}-\frac{r_{\s}(kr)_\s}{v(1-kr)}+Cr\\
& = -Cw+\frac{w_{\s\s}}{v}+\frac{w_\s(kr)_\s}{v(1-kr)} + C\varepsilon \tilde{r}\\
&\ge C\varepsilon \tilde{r} > 0
\,.
\end{align*}
The inequality for $r_1$ follows analogously by replacing $w$ with $r(\s,t)-\hat{r} = r(\s,t)-r_1e^{-C(1-\varepsilon)t}$.
Since the estimate is valid for all $\varepsilon>0$, we obtain the claimed result.
\end{proof}
This shows that the size of the graph function $|r|$ decays exponentially.
Next we prove a gradient bound.

\begin{lem}
Any solution $r$ to \eqref{PDE r}  satisfies
\[
|r_\s(\s,t)| 
\leq 
e^{\frac{c_1}{2}t}\sqrt{
    \sup\limits_{\s\in [0,L]}
        r_\s^2(\s,0)
    + \frac{c_2}{c_1}}
\]
where $c_1, c_2$ are constants that depend on the initial data $r_0$ and $\eta$ only.
\label{LMrsbd}
\end{lem}
\begin{proof}
Differentiate equation \eqref{PDE r} in $\s$ to see that
\begin{align*}
(r_\s)_t &= \frac{r_{\s\s\s}}{v}-\frac{v_\s r_{\s\s}}{v^2}-Cr_\s+r_{\s\s}\left(\frac{k_\s r+kr_\s}{v(1-kr)}\right)+r_\s\left(\frac{k_\s r+kr_\s}{v(1-kr)}\right)_\s
\\
&=\frac{r_{\s\s\s}}{v}-\frac{v_\s r_{\s\s}}{v^2}-Cr_\s+r_{\s\s}\left(\frac{k_\s r+kr_\s}{v(1-kr)}\right)
\\
&\quad
+ r_\s\left( \frac{k_{\s\s}r+2k_\s r_\s+kr_{\s\s}}{v(1-kr)}-\frac{(k_\s r+kr_\s)v_\s}{v^2(1-kr)}+\frac{(k_\s r+kr_\s)^2}{v(1-kr)^2}\right)\,.
\end{align*}
Substituting
\[
v_\s = -2(1-kr)(k_\s r+kr_\s)+2r_\s r_{\s\s}
\]
we have
\begin{align*}
r_{\s t}
&= 
    \frac{r_{\s\s\s}}{v}
    - Cr_\s
    - \frac{r_{\s\s}}{v^2}
        \left(
        -2(1-kr)(k_\s r+kr_\s)+2r_\s r_{\s\s}
        \right)
    + r_{\s\s}\left(\frac{k_\s r+kr_\s}{v(1-kr)}\right)
\\
&\qquad
+ r_\s\bigg(
    \frac{k_{\s\s}r+2k_\s r_\s+kr_{\s\s}}{v(1-kr)}
    + 2\frac{(k_\s r+kr_\s)^2}{v^2}
    - 2\frac{r_\s r_{\s\s}(k_\s r+kr_\s)}{v^2(1-kr)}
    + \frac{(k_\s r+kr_\s)^2}{v(1-kr)^2}
    \bigg)
\,.
\end{align*}
Arranging by the powers of $r_\s$ we find
\begin{align*}
r_{\s t}
- \frac{r_{\s \s \s }}{v}
&= 
    r_\s\left(
        - C
        + \frac{k_{\s \s }r}{v(1-kr)}
        + \frac{2k_\s^2r^2}{v^2}
        + \frac{k_\s^2r^2}{v(1-kr)^2}
    \right)
\\
&\quad
+ \frac{r_\s^2}{v}\left( \frac{2k_\s }{(1-kr)}   
+ \frac{4rkk_\s }{v}+\frac{2rkk_\s }{(1-kr)^2}\right)
+ \frac{r_\s^3}{v}
    \left( 
        \frac{2k^2}{v}+\frac{k^2}{(1-kr)^2}
    \right)
\\
&\quad
+ \frac{r_{\s \s }}{v}
  \left(
    \frac{2(1-kr)(k_\s r+kr_\s) - 2r_\s r_{\s\s}}{v}
    + \frac{k_\s r + kr_\s}{(1-kr)} 
    + \frac{kr_\s }{(1-kr)}
    - \frac{2r_\s (k_\s r+kr_\s )}{v(1-kr)} 
\right)
\,.
\end{align*}
Now, setting $Q = r_\s^2$, we have
\begin{align*}
Q_t & = 2r_\s r_{\s t}\\
Q_\s  &= 2r_\s r_{\s \s }\\
Q_{\s \s } &= 2(r_{\s \s })^2+2r_\s r_{\s \s \s }.
\end{align*}
Thus we have
\begin{align*}
Q_t
- \frac{Q_{\s \s }}{v}
&= 
2r_\s 
\left(
    r_{\s t}-\frac{r_{\s \s \s }}{v}
\right)
- \frac{2r_{\s\s}^2}{v}
\\&= 
2Q\left( 
    -C 
    + \frac{k_{\s \s }r}{v(1-kr)}
    + \frac{2(k_\s )^2r^2}{v^2}
    + \frac{(k_\s )^2r^2}{v(1-kr)^2}
    \right)
\\&\quad
+ \frac{2(r_\s)Q}{v}
    \left( 
    \frac{2k_\s }{(1-kr)}   
    + \frac{4rkk_\s}{v}
    + \frac{2rkk_\s}{(1-kr)^2}
    \right)
\\&\quad
+ \frac{2Q^2}{v}\left( \frac{2k^2}{v}
+ \frac{k^2}{(1-kr)^2}\right)
\\&\quad
+ \frac{Q_\s }{v}
    \left(
    - v_\s  
    + \frac{(k_\s r+kr_\s )}{v(1-kr)} 
    + \frac{k}{v(1-kr)}
    - \frac{2r_\s (k_\s r+kr_\s )}{v^2(1-kr)} 
\right)
\\&\quad
- \frac{2r_{\s \s }^2}{v}
\,.
\end{align*}
Since $v = (1-kr)^2 + Q$, we have $\frac{Q}{v} \leq 1$.
The uniform bounds for $r$ from Lemma \ref{LMrbd} imply $1/(1-kr)$ is uniformly bounded.
Thus
\[
Q_t - \frac{Q_{\s \s }}{v} 
\le 
    c_1Q
    + c_2
    + Q_{\s }b(k,k_\s,r,r_\s)
\] 
where $b$ is a continuous function and $c_i = c_i(\eta,r_0)$.

The maximum principle applied to the quantity
\[
X = e^{-c_1t}\left(
    Q+\frac{c_2}{c_1}
    \right)
\] 
yields the estimate
\[
|r_\s (\s,t)| 
\leq 
    e^{\frac{c_1}{2}t}
    \sqrt{\left(\sup\limits_{\s \in [0,L]}r_\s (\s ,0)\right)^2+\frac{c_2}{c_1}}
\,,
\]
as required.
\end{proof}

Lemma \ref{LMrsbd} proves part $(i)$ of Theorem \ref{TMmain}.


 \section{Existence and Uniqueness}
 \label{Sexistuniq}

Section \ref{Sapriori} gives a uniform bound on the gradient on a bounded time interval for any solution to \eqref{PDE r}.
In this section, we use results of Angenent \cite{angenent1990parabolic, angenent1991parabolic,angenent1990nonlinear} to obtain global existence, uniqueness, and smoothing.
We summarise this in the following.

\begin{prop}\label{PRglobalwp}
Let $\eta:\S\rightarrow \R^2$ be an embedded smooth (or analytic) curve, and $\gamma_0:\S\to\R^2$ a curve with $\gamma_0(\S)\subset T^\eta$, that is, normal graphical over $\eta$.
Assume further that $r_0\in h^{2+\alpha}([0,L(\eta)])$.

Then, the solution $\gamma:\S\times[0,\infty)\to\R^2$ to the target flow \eqref{EQflow} exists globally in time, is unique, and $\gamma(\cdot,t)$ is smooth (or analytic) for each $t\in(0,\infty)$. 
\end{prop}

The rest of this section is devoted to proving Proposition \ref{PRglobalwp}.
We begin by recalling the following theorem.

\begin{thm}[{\cite[Theorem 10.1]{angenent1990parabolic}}]
\label{TMang1}
Let $F:\mathbb{S}^1\times \mathbb{R}^3 \rightarrow \mathbb{R}$ satisfy
\begin{align*}
(F_1) &\text{ } F(x,r,p,q) \text{ is a locally Lipschitz function of its four arguments.}\\
(F_2) &\text{ }\lambda \leq F_q \leq \lambda^{-1} \\
(F_3) &\text{ }|F(x,r,p,0)| \leq \mu \\
(F_4) & \text{ }|F_x|+|F_r|+|qF_p| \leq \nu(1+|q|^2)
\end{align*}
where $\lambda,\mu$ and $\nu$ are constants.
Let $r_0$ be a given Lipschitz function. Then the problem
\begin{equation}
\begin{cases}r_t = F(\s,r,r_\s,r_{\s\s}) &\text{ in } \mathbb{S}^1\times(0,T),\\
r(\cdot, 0) = r_0 &\text{ on } \mathbb{S}^1
\end{cases}
\end{equation}
has a solution for any $T> 0$.
\end{thm}
We can apply the above theorem to \eqref{PDE r} .
With the above notation we set
\[
F(x,r,p,q) 
= 
    \frac{q}{(1-kr)^2+p^2}
    + \frac{p(kp+k_\s r)}{(1-kr)
        ((1-kr)^2+p^2)}
    -Cr
\,.
\]
Note that $F$ is smooth (or analytic) since $\eta$ is smooth (or analytic) and the a-priori estimates ensure $1-kr$ is uniformly bounded away from zero.

Next we calculate:
\begin{align*}
F_x 
&= 0
\\
F_r 
&= 
    \frac{2(1-kr)k}{((1-kr)^2+p^2)^2}q
    + \frac{pk_\s}{((1-kr)^2+p^2)(1-kr)}
    + \frac{kp(kp+k_\s r)}{(1-kr)^2((1-kr)^2+p^2)}
\\
&\quad
    + \frac{2kp(kp+k_\s r)}{((1-kr)^2+p^2)^2}
    - C
\\
F_p 
&= - \frac{2p}{((1-kr)^2+p^2)^2}q
    + \frac{2kp + k_\s r}{(1-kr)
        ((1-kr)^2+p^2)}
    - \frac{2p^2(kp+k_\s r)}{(1-kr)((1-kr)^2+p^2)^2}
\\
F_q 
&= \frac{1}{(1-kr)^2+p^2}
\,.
\end{align*}
The a-priori estimates imply that we may restrict the domain of $F$ to be a set where each of the conditions $(F_i)$ of Angenent's theorem hold.
Applying Angenent's theorem, we conclude existence of a solution $r:[0,L^\eta]\times[0,T]\to\R$ for any $T>0$, where $r_0$ is Lipschitz.

While existence only requires Lipschitz data, we need much more in order to obtain uniqueness.
The theorem that we shall apply is as follows:

\begin{thm}[{\cite[Theorem 2.7]{angenent1990nonlinear}}]
\label{T: ang}
Let $\mathcal{O} \subset E_1$ be open and $F:\mathcal{O} \rightarrow E_0$ be a $C^k$ map.
Consider the abstract initial value problem
\begin{equation}
\label{EQnonlin}
\begin{cases}
r_t = F(r(t)) & \text{ on } [0,T]\\
r(0)  = r_0 & \text{ for } t = 0.
\end{cases}
\end{equation}
If the Frechet derivative $dF$ belongs to $M_1(E)$ for all $r\in\mathcal{O}$ then the initial value problem \eqref{EQnonlin} has a unique (strict) solution on some small enough interval.
\end{thm}

The theorem essentially converts uniqueness of solutions to the linearised equation to uniqueness for the nonlinear equation \eqref{EQnonlin}.

Let us explain the notation used above. Set $E = (E_1,E_0)$ to be a Banach couple, where $E_1$ is densely included in $E_0$.
We put
\begin{align*}
X_1(E) = C([0,1];E_0)\quad\text{and}\quad
Y_1(E) = C([0,1];E_1)\cap C^1([0,1];E_0)
\,.
\end{align*}
Write $\partial_t$ for the bounded differentiation operator mapping from $Y$ to $X$.
For any operator $A \in \mathcal{L}(E_1,E_0)$ (where $\mathcal{L}(E_1,E_0)$ is the space of bounded, linear operators from $E_1$ to $E_0$), let $\hat{A}:Y_1(E)\rightarrow X_1\bigoplus E$ be defined by 
\[
\hat{A}u = (\partial_tu(t)-Au(t),u(0))
\,.
\]
Denote by $Hol(E)$ the set of $A\in \mathcal{L}(E_1,E_0)$ that generate an analytic semigroup $e^{tA}$.

Finally we define
\[
M_1(E)
= \Big\{ 
    A \in Hol(E)\ \vert\ \hat{A} \text{ is an isomorphism between } Y_1 \text{ and } X_1 \bigoplus E 
\Big\}
\,. 
\]

To apply the theorem we set $E_0 = h^\alpha(\mathbb{S}^1)$ and $E_1 = h^{2+\alpha}(\mathbb{S}^1)$.
These are little H\"older spaces, defined as the closure of $C^{\infty}(\mathbb{S}^1)$ in the usual $C^{k,\alpha}$ norm.
Here, we are using $\S = [0,L^\eta]$ with the endpoints identified (as in the rest of the paper).

Define the open set
\[
\mathcal{O} = \{ r\in E_1 : r \in I^\eta \} \subset E_1
\,.
\]
Set $F$ to be
\[
F(r) := F(\s,r,r_\s,r_{\s\s}) = \frac{r_{\s\s}}{v(r)} + \frac{r_\s(kr)_\s}{v(r)(1-kr)}
- Cr
\]
with
\[
v(r) = (1-kr)^2 + (r_\s)^2
\,.
\]
Note that we have conflated two uses of $F$: That of Theorem \ref{TMang1} and that of Theorem \ref{T: ang}.

Now we wish to show that $F$ is Fr\'echet differentiable. 
We first calculate the Gateaux derivative of $F$:
\[
G_{F}(h)
= \frac{d}{d \varepsilon}
    F(r+\varepsilon h)\bigg|_{\varepsilon=0}
\,. 
\]
For $G_F(h)$ to be the Fr\'echet derivative, it must satisfy
\begin{equation}
\label{EQfrelim}
\frac{|F(r+h) - F(r) - G_F(h)|}{||h||_{C^{2,\alpha}}} 
\rightarrow 0 \quad\text{ as }\quad ||h||_{C^{2,\alpha}} \rightarrow 0\,.
\end{equation}

We use the fact that Fr\'echet differentiable maps form an algebra, building our way up to $F$ slowly.

For the first term of $F$ the Gateaux derivative of $r\mapsto r_{\s\s}$ is
\[
G_{r_{\s\s}}(h) = h_{\s\s}
\,.
\]
The numerator of the Fr\'echet limit \eqref{EQfrelim} of $r_{\s\s}$ is zero, and so $r\mapsto r_{\s\s}$ is Fr\'echet differentiable.

Next we calculate the Gateuax derivative of $r\mapsto v(r)$:
\[
\frac{d}{d\varepsilon}v[r+\varepsilon h]\bigg{|}_{\varepsilon = 0} = -2kh(1-kr)+2r_\s h_\s
\,.
\]
Therefore
\begin{align*}
v(r+h) 
- v(r)
- G_{v}(h) 
&= 
    (1-kr-kh)^2
    + (r_\s + h_\s)^2
    - (1-kr)^2
\\&\qquad
    - (r_\s)^2
    - 2r_\s h_\s
    + 2kh(1-kr)
\\&=
    k^2h^2
    + h_\s^2 
\le
    C||h||^2_{C^{2,\alpha}}
\,.
\end{align*}
Above $C$ is a constant dependent on $k=k^\eta$, which is uniformly bounded. Hence $r\mapsto v(r)$ is also Fr\'echet differentiable.

It is now clear that the first fraction of $F(r)$ is Fr\'echet differentiable.
The second term consists of three parts: the numerator $r\mapsto r_\s(kr)_\s$, and the product of $r\mapsto v(r)$, $r\mapsto 1-kr$ in the denominator.
Since $r\mapsto 1-kr$ is clearly Fr\'echet differentiable, we only check the numerator.
We find
\[
G_{r_\s(kr)_\s}(h)
 =
    h_\s(kr)_\s
    + r_\s(k_\s h + k h_\s)
\]
and then the numerator of the Fr\'echet limit \eqref{EQfrelim} is
\[
(r+h)_\s(kr+kh)_\s
- r_\s(kr)_\s
- (
    h_\s(kr)_\s
    + r_\s(k_\s h + k h_\s)
)
= h_u(kh)_u
\le C||h||_{C^{2,\alpha}}^2
\,.
\]
Again the constant $C$ depends on $k$ and $k_\s$.

Thus the second term in $F$ is also Fr\'echet differentiable.
Finally, the term $-Cr$ is also clearly Fr\'echet differentiable.
This settles the Fr\'echet differentiability of $F$.

Now we can apply standard parabolic theory (for example \cite[Theorem 5.6]{lieberman1996second}) to see that the linear PDE $h_t = G_{F(r)}(h)$ has a unique solution. 
Indeed, since $r_0 \in h^{2+\alpha}$ the coefficients of $G_{F(r)}(h)$ are in $h^\alpha \subset C^{0,\alpha}$ (which is $H^\alpha$ in the notation of \cite{lieberman1996second}).
In particular, we conclude that the Fr\'echet derivative of $F$ is in the space $M_1(E)$.

Applying Theorem \ref{T: ang} shows that the solution generated by Theorem \ref{TMang1} is unique.
If not, then there must exist a $T_0\in[0,\infty)$ such that there are at least two distinct solutions with initial data given by $r_0 = r(\cdot,T_0)$.
This is in contradiction with Theorem \ref{T: ang}.

Now let us turn to the smoothing effect.
This time it is a straightforward application of \cite[Corollary 2.10]{angenent1990nonlinear}, which does not require checking any further assumptions.
The conclusion is that
\[
    t^mr^{(m)}(t) \in C^0([0,T];E_1)
\]
for all $t>0$ and all $T$.
The above regularity holds for all $m$, guaranteeing that $r(t)$ is smooth for $t>0$.
Furthermore, the same result gives
\[
||t^mr^{(m)}(t)||_{E_1} \le C\frac{M^m}{m!}
\]
if $\eta$ is analytic, giving that $r(\cdot,t)$ is analytic for $t>0$ in this case.

This concludes the proof of Proposition \ref{PRglobalwp}, as well as the unique global existence and smoothing parts of Theorem \ref{TMmain}.


\section{Smooth Convergence}
\label{Ssmoothconv}

The goal of this section is to complete the proof of Theorem \ref{TMmain}.
It remains to establish convergence of $\gamma$ to $\eta$ in the smooth topology, which is equivalent to convergence of $r(\cdot,t)$ to zero in the smooth topology on $[0,L^\eta]$.

Currently, our gradient bound Lemma \ref{LMrsbd}, is not uniform.
The key observation is that if $r$ is small, a uniform bound \emph{is} available.
We can always guarantee that $r$ will be eventually arbitrarily small, as the barriers from Lemma \ref{LMrbd} converge to zero.
Indeed, if we set $T_{\varepsilon}>0$ to be such that 
\begin{equation}
\label{EQTeps}
r(\s,t) \in (-\varepsilon,\varepsilon)\quad\text{for all $s\in[0,L^\eta]$ and $t > T_{\varepsilon}$}
\,,
\end{equation}
then Lemma \ref{LMrbd} implies that $T_\varepsilon < (1/C)\log(\max\{|r_1|,r_2\}/\varepsilon)$.

The uniform gradient estimate is as follows.

\begin{prop}\label{pro: rs}
Let $\gamma:\S\times[0,\infty)\to\R^2$ be the target flow solution given by Proposition \ref{PRglobalwp}, and $r:\S\times[0,\infty)\to\R$ its graph function.
There is an $\varepsilon_0=\varepsilon_0(\eta)$ and $C_0=C_0(\eta)$ such that the following holds.
For $t>T_{\varepsilon_0}$ (as in \eqref{EQTeps}) and $C\ge C_0$ we have
\[|r_\s(\s,t)|
    \leq 
    e^{\frac{1}{2}(T_{\varepsilon_0}-t)}
    \max\limits_{\s\in [0,L]}|r_\s(T_{\varepsilon_0})|
    \,.
  \]
\end{prop}
\begin{proof}
Recall the proof of Lemma \ref{LMrsbd}:
For $t>T_\varepsilon$ the evolution of $Q = r_\s^2$ can be estimated with
\begin{align*}
Q_t
- \frac{Q_{\s\s}}{v}
&\leq 
    2Q\bigg(
        - C
        + c_3\varepsilon 
        + \frac{2k_\s r_\s}{v(1-kr)}
        + \frac{2k^2}{v}
        + \frac{k^2}{(1-kr)^2}
    \bigg)
\\&\quad
    + \frac{Q_\s}{v}\left(
        - v_\s
        + \frac{(k_\s r + kr_\s)}{v(1-kr)} 
        + \frac{k}{v(1-kr)}
        - \frac{2r_\s(k_\s r + kr_\s)}{v^2(1-kr)} 
    \right)
\end{align*}
where the constant $c_3$ depends only on $\eta$ via $k$, $k_\s$ and $k_{\s\s}$.
Furthermore, as $|r_\s/v| \le 1$, we may estimate
\[
\frac{2k_\s r_\s}{v(1-kr)}
        + \frac{2k^2}{v}
        + \frac{k^2}{(1-kr)^2}
\le c_4
\]
where $c_4$ depends only on $\eta$.

Now we choose $\varepsilon_0$ and place a restriction on $C$.
In the case of non-convex $\eta$, choose $\varepsilon_0 = \max\{1/|k_{min}|,1/k_{max}\}$.
If $\eta$ is convex, choose $\varepsilon_0 = 1/k_{max}$.
Then, we require $C \ge C_0 := c_3\varepsilon_0+c_4+1$.

The evolution equation can thus be estimated by
\[
Q_t
- \frac{Q_{\s\s}}{v} 
\leq 
    -Q 
    + Q_\s\left(
        - v_\s
        + \frac{(k_\s r + kr_\s)}{v(1-kr)} 
        + \frac{k}{v(1-kr)}
        - \frac{2r_\s(k_\s r+kr_\s)}{v^2(1-kr)} 
    \right)
    \,.
\]
Let $X = e^{t}Q$.
Then
\[ 
X_t
- \frac{X_{\s\s}}{v} 
\leq X_\s\left(
        - v_\s
        + \frac{(k_\s r + kr_\s)}{v(1-kr)} 
        + \frac{k}{v(1-kr)}
        - \frac{2r_\s(k_\s r+kr_\s)}{v^2(1-kr)} 
    \right)
\,.
\]
By the maximum principle we have for $t >T_{\varepsilon_0}$
\[|r_\s(\s,t)|
\leq 
    e^{\frac{1}{2}(T_{\varepsilon_0}-t)}
    \max\limits_{\s\in [0,L]}|r_\s(T_{\varepsilon_0})|
\,,
\]
as required.
\end{proof}

Proposition \ref{pro: rs} implies that $r_\s$ becomes arbitrarily small.
We will use this control in addition to our established control on $r$ to obtain control over $r_{\s\s}$ below.

\begin{prop}
Let $\gamma:\S\times[0,\infty)\to\R^2$ be the target flow solution given by Proposition \ref{PRglobalwp}, and $r:\S\times[0,\infty)\to\R$ its graph function.
There is a $t_0 = t_0(r_0,\eta)$ such that the following holds.
For $t>\max\{t_0,T_{\varepsilon_0}\}$ and $C\ge \max\{1,C_0\}$ ($T_{\varepsilon_0}$ and $C_0$ as in Proposition \ref{pro: rs}) we have
\[
r_{\s\s}^2(\s,t)
\le \Big(\max r^2_{\s\s}(\cdot,t_0) + 2c_0e^{-t_0}\Big) e^{(t_0-t)C/2}
\]
where $c_1 = \max r^2_{\s\s}(\cdot,t_0)$.
\label{PRrssunif}
\end{prop}
\begin{proof}
Recall that
\[
(r_\s)_t = \frac{r_{\s\s\s}}{v}-\frac{v_\s r_{\s\s}}{v^2}-Cr_\s+r_{\s\s}\left( \frac{(kr)_\s}{v(1-kr)}\right)+r_\s\left( \frac{(kr)_\s}{v(1-kr)}\right)_\s
\]
so
\begin{align*}
    (r_{\s\s})_t = &\frac{r_{\s^4}}{v}-2\frac{v_\s r_{\s\s\s}}{v^2}-\frac{v_{\s\s}r_{\s\s}}{v^2}+\frac{2(v_\s)^2r_{\s\s}}{v^3}-Cr_{\s\s}\\
    &+r_{\s\s\s}\left( \frac{(kr)_\s}{v(1-kr)}\right)+2r_{\s\s}\left( \frac{(kr)_\s}{v(1-kr)}\right)_\s+r_\s\left( \frac{(kr)_\s}{v(1-kr)}\right)_{\s\s}.
\end{align*}
Now
\begin{align*}
    \left( \frac{(kr)_\s}{v(1-kr)}\right)_\s
    &= 
        \frac{(kr)_{\s\s}}{v(1-kr)}
        -\frac{(kr)_\s v_\s}{v^2(1-kr)}
        +\frac{((kr)_\s)^2}{v(1-kr)^2}
\,,\text{ so }
\\
     \left( \frac{(kr)_\s}{v(1-kr)}\right)_{\s\s}
&= \frac{(kr)_{\s\s\s}}{v(1-kr)}
- 2\frac{(kr)_{\s\s}v_\s}{v^2(1-kr)}
+ \frac{(kr)_{\s\s}(kr)_\s}{v(1-kr)^2}
- \frac{(kr)_\s v_{\s\s}}{v^2(1-kr)}
\\&\quad
    + 2\frac{(kr)_\s v^2_\s}{v^3(1-kr)}
    -\frac{2((kr)_\s)^2v_\s}{v^2(1-kr)^2}
    + \frac{2(kr)_\s(kr)_{\s\s}}{v(1-kr)^2}
    + 2\frac{((kr)_\s)^3}{v(1-kr)^3}.
\end{align*}
We will also use
\begin{align*}
v_{\s} 
&= -2(1-kr)(kr)_\s+2(r_\s)(r_{\s\s})
\\
v_{\s\s} 
&= 
    - 2(1-kr)(kr)_{\s\s}
    + 2((kr)_\s)^2
    + 2r_{\s\s}^2
    + 2r_\s r_{\s\s\s}
\\
v_{\s}^2 
&= 
    4(1-kr)^2(kr)_\s^2
    - 8(1-kr)(kr)_\s r_\s r_{\s\s}
    + 4r_\s^2r_{\s\s}^2
    \,.
\end{align*}
We write the evolution for $r_{\s\s}$ grouping by powers of $r_{\s\s}$.
\begin{align*}
(r_{\s\s})_t &- \frac{(r_{\s\s})_{\s\s}}{v}
\\&=  
    (r_{\s\s})^3
    \left[ 
        \frac{-2}{v^2}
        + \frac{8(r_\s)^2}{v^3}
    \right]
\\&\quad
    + (r_{\s\s})^2\bigg[ 
        \frac{-2(1-kr)k}{v^2}
        + \frac{2}{v^2}(-8(1-kr)(kr)_\s r_\s
        + \frac{2k}{v(1-kr)}
\\&\qquad
    - \frac{(kr)_\s 2r_\s}{v^2(1-kr)}
    - \frac{-4k(r_\s)^2}{v^2(1-kr)}
    + \frac{(kr)_\s kr_\s}{v(1-kr)}
\\&\qquad
    - \frac{-2(kr)_\s r_\s}{v^3(1-kr)}
    + \frac{8(kr)_\s(r_\s)^3}{v^2(1-kr)}
    + \frac{(kr)_\s kr_\s}{v(1-kr)^2}
    \bigg]
\\&\quad
    + r_{\s\s}\bigg[ 
        \frac{2(1-kr)(k_{\s\s}r
        + 2k_\s r_\s)
        - 2(kr)_\s^2
        + 8(1-kr)(kr)_\s^2}{v^2}
        -C
\\&\qquad 
        + 2\left(\frac{k_{\s\s}r
        + 2k_\s r_\s}{v(1-kr)}
        + \frac{4(k_{\s\s}r + 2k_\s r_\s)(1-kr)(kr)_\s}{v^2(1-kr)}
        + \frac{(kr)_\s^2}{v(1-kr)^2}\right)
\\&\qquad
    + r_\s\bigg( 
        \frac{3k_\s}{v(1-kr)}
        - \frac{4r_\s (k_{\s\s}r + 2k_\s r_\s)}{v^2(1-kr)}
        + \frac{k(kr)_\s}{v(1-kr)^2}
\\&\qquad 
    + \frac{2(1-kr)k(kr)_\s}{v^2(1-kr)}
    - \frac{4((kr)_\s)^2(r_\s)}{v^2(1-kr)}
    + \frac{2(kr)_\s k}{v(1-kr)^2}
    \bigg) 
\bigg]
\\&\quad
    + r_{\s\s\s}\,f_4(k,k_\s,r,r_\s)
    + f_5(k,k_\s,k_{\s\s},k_{\s\s\s},r,r_\s)
\\&=  
    (r_{\s\s})^3\,f_1(k,r,r_\s)
    + (r_{\s\s})^2\,f_2(k,k_\s,r,r_\s)
    + r_{\s\s}\,f_3(k,k_\s,k_{\s\s},r,r_\s)
\\&\quad
    + r_{\s\s\s}\,f_4(k,k_\s,r,r_\s)
    + f_5(k,k_\s,k_{\s\s},k_{\s\s\s},r,r_\s)
    \,.
\end{align*}
As $r$, $r_\s$ tend to zero, the functions $f_2$, $f_4$ and $f_5$ tend to zero, $f_3$ tends to $-C$, and $f_1$ tends to $-2$.
Due to the estimates Lemma \ref{LMrbd} and Proposition \ref{pro: rs}, this is guaranteed as $t\to\infty$, at an exponential rate ($r$ decays with rate $e^{-Ct}$, and $r_\s$ with rate $e^{-\frac12t}$).
So, for $t>T_\delta$ we have (here $c_0=c_0(r_0,\eta)$ and we have assumed $C\ge1$)
\begin{align*}
\frac12(r_{\s\s}^2)_t - \frac{(r_{\s\s}^2)_{\s\s}}{v}
& \le  
    r_{\s\s}^4\,(-2 + c_0 e^{-t})
    + |r_{\s\s}|^3\,( c_0 e^{-t})
\\&\qquad\qquad
    + r_{\s\s}^2\,(-C+c_0 e^{-\frac12 t})
    + \frac12((r_{\s\s})^2)_{\s}\,f_4(t)
    + |r_{\s\s}|c_0 e^{-t}
    \,.
\end{align*}
or
\[
\frac12(r_{\s\s}^2)_t - \frac{(r_{\s\s}^2)_{\s\s}}{v}
    - \frac12((r_{\s\s})^2)_{\s}\,f_4(t)
\le
      r_{\s\s}^4\,(-2 + 2c_0 e^{-t})
    + r_{\s\s}^2\,(-C + c_0 e^{-\frac12 t} + (c_0/2)e^{-t})
    + c_0 e^{-t}
\]
where we have estimated $|r_{\s\s}|^3 \le r_{\s\s}^4 + (1/4)r_{\s\s}^2$ and $|r_{\s\s}| \le (1/4)r_{\s\s}^2 + 1$.

Finally, by taking $t$ larger than $t_0=\max\{\log(c_0),\log(4c_0/C),\log(16c_0^2/C^2)\}$, we have
\[
\frac12(r_{\s\s}^2e^{tC/2})_t - \frac{(r_{\s\s}^2e^{tC/2})_{\s\s}}{v}
    - \frac12((r_{\s\s})^2e^{tC/2})_{\s}\,f_4(t)
\le
    c_0 e^{-t}
    \,,
\]
or
\[
\frac12(r_{\s\s}^2e^{tC/2} + 2c_0e^{-t})_t
- \frac{(r_{\s\s}^2e^{tC/2} + 2c_0e^{-t})_{\s\s}}{v}
    - \frac12((r_{\s\s})^2e^{tC/2} + 2c_0e^{-t})_{\s}\,f_4(t)
\le
    0
    \,.
\]

and so we obtain the estimate
\[
r_{\s\s}^2
\le \Big(\max r^2_{\s\s}(\cdot,t_0) + 2c_0e^{-t_0}\Big) e^{(t_0-t)C/2}
\]
as required.
\end{proof}

Let us now conduct a standard bootstrapping procedure, which we briefly summarise.
Proposition \ref{PRrssunif} implies that $r_\s \in C^{\alpha}$.
The PDE for $r_\s$ is 
\begin{equation}\label{general}
(r_\s)_t
= a(k,r,r_\s)r_{\s\s\s}
  + r_{\s\s}b(k_{\s^i},r,r_\s)
  + c(k_{\s^i},r,r_\s)r_\s
\,,
\end{equation}
where the $i$ in $k_{\s^i}$ ranges from $0$ to $2$.
Applying a theorem from Lieberman \cite[Theorem 5.14]{lieberman1996second} then yields a uniform $C^{2,\alpha}$ estimate for $r_\s$.
This then implies that $r_{\s\s\s}\in C^\alpha$.
Repeating the process shows that there is a uniform $C^\alpha$ bound for an arbitrary number $m$ of derivatives $r_{\s^m}$.
Furthermore interpolation combined with our exponential decay estimates yields the desired smooth convergence.

\begin{lem}
Let $\gamma:\S\times[0,\infty)\to\R^2$ be the target flow solution given by Proposition \ref{PRglobalwp}, and $r:\S\times[0,\infty)\to\R$ its graph function.
There exist constants
such that the following holds.
\[ 
r_{\s^m} 
\le 
    C_me^{-c_mt}
\]
for all $m\in\mathbb{N}$.
\label{LMcinfconv}
\end{lem}
\begin{proof}
By periodicity of $r$ the average of $r_{\s^m} = 0$ (for $m\ge1$), so there exists a point $\s^* \in [0,L]$ such that $r_{\s^m}(\s^*) = 0$.
Then the FTC, integration by parts, the H\"older inequality, our uniform bounds and exponential decay estimate implies
\begin{align*}
||r_{\s^m}||_{L^\infty(d\s)} &\leq \int_0^L|r_{\s^{m+1}}|d\s\\
    &\leq \sqrt{L}\left( \int_0^L(r_{\s^{m+1}})^2d\s\right)^{\frac{1}{2}}\\
    & \leq \sqrt{L}\left( \int_0^L|r|\,|r_{\s^{2m+2}}|\,d\s\right)^{\frac{1}{2}}\\
    &\leq L||r(\cdot,0)||_{L^\infty(d\s)}e^{-\frac{Ct}{2}}||r_{\s^{2m+2}}||_{L^\infty(d\s)}.
\end{align*}
where we have used the Holder inequality and integration by parts. 
\end{proof}

With Lemma \ref{LMcinfconv}, we have completed the proof of Theorem \ref{TMmain}.


\printbibliography

\end{document}